\documentclass{amsproc}
\usepackage[T1]{fontenc}
\usepackage{hyperref}
\hypersetup{
    colorlinks=true,
    linkcolor=blue,
    urlcolor=blue,
    citecolor=blue
}
\usepackage{verbatim}
\usepackage{tgpagella}
\usepackage{xcolor}
\usepackage{tikz-cd}
\usepackage{enumitem}
\usepackage{amsmath,amssymb,amsthm,mathtools,mathrsfs,amsfonts}

\newtheorem{theorem}{Theorem}
\newtheorem{example}[theorem]{Example}
\newtheorem{lemma}[theorem]{Lemma}
\newtheorem{remark}[theorem]{Remark}
\newtheorem{definition}[theorem]{Definition}

\newtheorem{conjecture}[theorem]{Conjecture}
\newtheorem{proposition}[theorem]{Proposition}

\def\AA{\mathbb{A}}
\def\PP{\mathbb{P}}

\def\sF{\mathscr{F}}
\def\sG{\mathscr{G}}

\def\KK{\mathbb{K}}
\def\N{\mathcal{N}}

\def\V{\mathcal{V}}

\def\ZZ{\mathbb{Z}}
\def\NN{\mathbb{N}}
\def\CC{\mathbb{C}}
\DeclareMathOperator{\cha}{char}

\DeclareMathOperator{\ara}{ara}
\DeclareMathOperator{\cd}{cd}
\DeclareMathOperator{\chara}{char}
\DeclareMathOperator{\Spec}{Spec}
\DeclareMathOperator{\Proj}{Proj}
\DeclareMathOperator{\Tor}{Tor}
\DeclareMathOperator{\codim}{codim}
\DeclareMathOperator{\ecd}{\acute{e}cd}

\DeclareMathOperator{\height}{height}
\DeclareMathOperator{\qccd}{qccd}

\newcommand{\citeproject}[1]{%
  \cite[Tag~\href{https://stacks.math.columbia.edu/tag/#1}{#1}]{stacks-project}%
}

\begin{document}

\title{\'Etale and Quasicoherent Cohomological Dimensions of Subspace Arrangements}

\author{Manolis C. Tsakiris \, \, \, \, \, \, Matteo Varbaro}

\address[M.Tsakiris]{State Key Laboratory of Mathematical Sciences, Academy of Mathematics and Systems Science, Chinese Academy of Sciences, 100190, Beijing, China}
\email{manolis@amss.ac.cn}

\address[M. Varbaro]{Dipartimento di Matematica,  Universit\`a di Genova, Via Dodecaneso 35, 16146 Genova, Italy}
\email{Matteo.Varbaro@unige.it}

\maketitle


\begin{abstract}
We report some work in progress on the relationship between \'etale and quasi-coherent cohomological dimensions. 
\end{abstract}

\section{Introduction}

Let $U$ be a Noetherian scheme. Given a sheaf $\sF$ on the Zariski (respectively, \'etale) site of $U$, $H^i(U,\sF)$ (respectively, $H_{\text{\'et}}^i(U,\sF)$) denotes the $i$th sheaf cohomology of $\sF$. Recall that the {\it quasi-coherent cohomological dimension} of $U$ is
\[\qccd(U)=\sup\{r\in\NN:H^r(U,\sF)\neq 0 \mbox{ for some quasi-coherent sheaf $\sF$}\},\]
while the {\it \'etale cohomological dimension} of $U$ is
\[\ecd(U)=\sup\{r\in\NN:H^r_{\text{\'et}}(U,\sF)\neq 0 \mbox{ for some torsion sheaf $\sF$}\}.\]
As every torsion sheaf is the direct limit of constructible sheaves \cite[Proposition 5.8.8]{LeiFu}, in the definition above we may replace "torsion" by "constructible".

The notions of \'etale and quasi-coherent cohomological dimensions of a Noetherian scheme are known to provide measures for both its topological and arithmetical complexity \cite{lyubeznik1993etale,huneke1990vanishing}. For example, if $I$ is an ideal of a finitely generated $\KK$-algebra $A$, where $\KK$ is a separably closed field, its arithmetical rank
\[\ara(I)=\min\{r\in\NN: \exists \ f_1,\ldots ,f_r\in A \mbox{ with }\sqrt{I}=\sqrt{(f_1,\ldots ,f_r)}\}\]
is bounded below by these cohomological dimensions. Precisely, if we are in either of the following situations,
\begin{enumerate}
\item $U=\Spec(A)\setminus \V(I)$;
\item $U=\Proj(A)\setminus \V_+(I)$ if $A$ is standard graded, $A_0=\KK$ and $I$ is graded,
\end{enumerate}
we have
\[\ara(I)\geq \qccd(U)+1, \ \ \ \ara(I)\geq \ecd(U)+1-\dim(U).\]
In the affine case, the first inequality simply follows by computing sheaf cohomology as \v Cech cohomology, since $U$ can be covered by $\ara(I)$ many affines subsets. As explained in \cite{bruns1990number}, the second inequality follows again since $U$ can be covered by $\ara(I)$ many affine subsets, inductively using the Mayer-Vietoris sequence and Artin's theorem saying that the \'etale cohomological dimension of an affine scheme of finite type over a separably closed field coincides with its dimension, see Theorem 7.2 of Chapter VI in \cite{milne1980etale}. In the graded setting the above strategies work in the same way to prove the inequalities additionally requiring that the elements giving the arithmetical rank are homogeneous. Alternatively, one can avoid to assume this because, if $A$ is standard graded, $A_0=\KK$ and $I$ is homogeneous:

\begin{enumerate}
\item $\qccd \big(\Spec(A)\setminus \V(I)\big)=\qccd\big(\Proj(A)\setminus \V_+(I)\big)$, see \S \ref{subsection:QCD};
\item $\ecd\big(\Spec(A)\setminus \V(I)\big)=\ecd\big(\Proj(A)\setminus \V_+(I)\big)+1$, see Proposition \ref{prp:ecd-affine-cone}.  
\end{enumerate}

In the few examples where $\ara(I)$ has been computed, the second lower bound above has always been at least as good as the first, and this led Lyubeznik to state the following conjecture in \cite[Section 15]{Lyu02}: 

\begin{conjecture}\label{conj:lyu}
Let $U$ be a scheme of finite type over a separably closed field $\KK$. Then
\begin{align*}
\ecd(U) \ge \dim(U) + \qccd(U).  
\end{align*}
\end{conjecture} \noindent In the generality above, Lyubeznik's conjecture is not true, as the following example, pointed out to us by Emiliano Ambrosi, shows:

\begin{example}\label{counterexamplelyu}
There exist algebraic varieties $U$ over $\KK=\CC$ that are not affine but such that their associated analytic space is Stein. One such example is provided by Serre (e.g., see Example 3.2 in Chapter VI of \cite{ASAV}). In this example, $U$ is a non-empty Zariski-open subset of a smooth projective surface.
Since $U$ is not affine, $\qccd(U)\geq 1$ by Serre's criterion; however, the associated analytic space of $U$ being Stein, $\ecd(U)=\dim(U)=2$ from combining \cite[Expos\'e XVI, Th\'eor\`eme 4.1]{SGA4-3} and \cite[Subsection 1.1.2]{benoist25}.
\end{example}

\begin{remark}
Remarkably, motivated by the vanishing theorems of Artin \& Grothendieck for \'etale cohomology, the fact that the de-Rham cohomological dimension of a non-singular scheme is always bounded from above by $\qccd(U)+\dim(U)$ (see \cite[Proposition III.7.2]{ASAV}), Hartshorne had asked whether the opposite inequality to the one later conjectured by 
Lyubeznik, that is $\ecd(U) \le \qccd(U) + \dim(U)$, is true (\cite[Problem IV.4.1]{ASAV}). In turn, this prompted Ogus to give a counterexample (see (1) below), which thus satisfies Lyubeznik's Conjecture \ref{conj:lyu}. 
\end{remark}

Conjecture \ref{conj:lyu} has been shown to hold true in several situations; for example, whenever $U$ is a Zariski-open subset of $\AA^N$ or $\PP^N$, so that:
\begin{enumerate}
    \item $\KK$ has characteristic 0 and $X=\PP^N\setminus U$ is a $d$-fold Veronese embedding of $\PP^n$ with $d\geq 2$ and $N={n+d \choose d}-1$, \cite{ogus1973local}. Indeed, if $n\geq 2$ the conjectured inequality is strict: $\qccd(U)=N-n-1$ while $\ecd(U)\geq 2N-2$.
    \item $X=\AA^N\setminus U$ is a determinantal variety, \cite{bruns1990number}. Indeed, if $N=mn$, $X$ consists of the matrices of size $m\times n$ and of rank less than $t$, the inequality is strict as soon as $t>1$, $\{t,m,n\}\neq \{t\}$  and $\KK$ has positive characteristic. In this case $\qccd(U)=(m-t+1)(n-t+1)-1$ while $\ecd(U)=2N-t^2$.      
\item $X=\AA^N\setminus U$ is a Pfaffian variety, \cite{Barile95}. Indeed, if $N=n(n-1)/2$, $X$ consists of the alternating matrices of size $n\times n$ and of rank less than $t$ (with $t$ even), the inequality is strict as soon as $2<t<n$ and  $\KK$ has positive characteristic: in this case, $\qccd(U)=\frac{(n-t+1)(n-t+2)}{2}-1$ while $\ecd(U)=2N-t(t-1)/2$. 
    \item $X=\AA^N\setminus U$ is a symmetric determinantal variety, \cite{Barile95}. Indeed, if $N=n(n+1)/2$, $X$ consists of the symmetric matrices of size $n\times n$ and of rank less than $t$, the inequality is strict as soon as $2<t<n$ and  $\KK$ has positive characteristic: in this case, $\qccd(U)=\frac{(n-t+1)(n-t+2)}{2}-1$ while:
    \begin{itemize}
    \item $\ecd(U)=2N-t(t+1)/2$ if $\cha(\KK)\neq 2$ (in this case the inequality is strict even for $t=2$);
    \item If $\cha(\KK)= 2$, $\ecd(U)=2N-t(t+1)/2$ if $t$ is odd, while $\ecd(U)=2n(n-1)/2-t(t-1)/2$ if $t$ is even.
    \end{itemize}
    Twenty years after \cite{Barile95}, the quasi-coherent cohomological dimension had been computed also in the case in which $t$ is even and $\cha(K)=0$ in \cite{RW16}$: \qccd(U)=n(n-1)/2-t(t-1)/2$; since in this case $\ecd(U)=2N-t(t+1)/2$ by \cite{Barile95}, the inequality is strict also in characteristic 0 and $t$ even. Note that, if $t=2$, this coincides with (1) for $d=2$ (up to taking affine cones).
    \item $\KK$ has characteristic 0 and $X=\PP^N\setminus U$ is smooth, further assuming that $\qccd(U)>\codim_{\PP^N}(X)-1$, \cite{varbaro2012arithmetical}.
\end{enumerate} 

In this paper, which reports work in progress, we yet prove Conjecture \ref{conj:lyu} for a family of (complements of) subspace arrangements in arbitrary characteristic. 

\subsection*{Acknowledgements}
We are highly indebted to Emiliano Ambrosi for suggesting Example \ref{counterexamplelyu}. Moreover, we are grateful to Emiliano Ambrosi, Lei Fu and Shizhang Li for enlightening discussions about \'etale cohomology.

\bigskip

\section{Preliminaries}

\subsection{Quasicoherent Cohomological Dimension} \label{subsection:QCD}
The quasicoherent cohomological dimension $\qccd(X)$ of a scheme $X$, is the supremum over all integers $j$ such that there exists a quasicoherent sheaf $\sF$ on $X$ with $H^j(X,\sF) \neq 0$. If $Y$ is a closed subscheme of $X$, then the local quasicoherent cohomological dimension $\qccd(X,Y)$ of $X$ with respect to $Y$ is the supremum over all integers $j$ such that there exists a quasicoherent sheaf $\sF$ on $X$ with $H^j_Y(X,\sF) \neq 0$. If $X$ is affine, say $X = \Spec(A)$, and $Y = \Spec(A/I)$ for some ideal $I$ of $A$, then $\qccd(X,Y)$ is also denoted by $\cd(A,I)$, and it is equal to the largest integer $j$ such that the local cohomology module $H_I^j(A)$ is non-zero. The local cohomology long exact sequence
\begin{align*}
\cdots \rightarrow H^{j}(X,\sF) \rightarrow H^{j}(X\setminus Y,\sF) \rightarrow H^{j+1}_Y(X,\sF) \rightarrow H^{j+1}(X,\sF) \rightarrow \cdots,
\end{align*} together with Serre's affineness criterion, which says that $X$ is affine if and only if $\qccd(X)$ is zero, yields, when $X = \Spec(A)$ and $\cd(A,I)>0$, that $$\qccd(X \setminus Y) = \qccd(X,Y)-1 =  \cd(A,I)-1.$$

Now, let $Y \subseteq X \subseteq \PP^n_\KK$ be projective schemes over a field $\KK$ and denote by $Y' \subseteq X' \subseteq \AA^{n+1}_\KK$ their respective affine cones. Consider the open set  $U = X \setminus Y$ of $X$ and its affine cone $U' = X' \setminus Y'$. Let $S$ be a polynomial ring over $\KK$ in $n+1$ variables equipped with the standard grading and $I \subseteq J$ be homogeneous ideals such that $X = \Proj(S/I)$ and $Y = \Proj(S/J)$. Then for any quasicoherent sheaf $\sF$ on $X$ there is a well-known isomorphism of $\mathbb{Z}$-graded $\KK$-vector spaces (e.g., see \cite[\S 2]{Hartshorne-CDAV})
\begin{align*}
\bigoplus_{\nu \in \mathbb{Z}} H^i\big(U, \sF(\nu) \big) \cong H_J^{i+1}(M)
\end{align*} for every $i \ge 1$, where $M$ is the graded $S/I$-module such that $\sF = \tilde{M}$. From this it follows that $\qccd \big(U) = \qccd \big(U' \big)$.

\subsection{\'Etale Cohomological Dimension} \label{subsection:ECD}

For a scheme $X$ and a prime number $\ell$, the $\ell$-adic \'etale cohomological dimension $\ecd_\ell(X)$ of $X$ is defined as the largest integer $j$ such that $H_{\text{\'et}}^j(X,\sF) \neq 0$, where $\sF$ is an $\ell$-torsion sheaf $\sF$ on the \'etale site of $X$. Then $\ecd(X)$ is the supremum of all $\ecd_\ell(X)$'s as $\ell$ varies across all primes.

In the rest of this manuscript we will be concerned with schemes that are quasi-projective over a separably closed field $\KK$. For such a scheme $X$, it is known that $\ecd_\ell(X) \le \dim(X)$ when $\ell = \chara(\KK)$ \cite[Corollaire 5.2]{SGA4-3}. On the other hand, it follows from the theorems of Galois cohomology (e.g., \cite[Chapter II, Proposition 11]{Serre-Galois}) and the Leray spectral sequence associated to the canonical morphism $\Spec\big(\kappa(\xi)\big) \rightarrow X$, where $\xi \in X$ is the generic point of an irreducible component of $X$ of maximal dimension, that $\ecd_\ell(X) \ge \dim(X)$ for $\ell \neq \chara(\KK)$. Hence, the supremum above may be restricted to all primes $\ell \neq \chara(\KK)$. 

If $Y$ is a closed subscheme of $X$, the local \'etale cohomological dimension of $X$ with respect to $Y$ is defined as the largest integer $j$ such that $H_{\text{\'et},Y}^j(X,\sF) \neq 0$. There is again a local cohomology long exact sequence reading 
\begin{align*}
\cdots \rightarrow H^{j}_{\text{\'et}}(X,\sF) \rightarrow H^{j}_{\text{\'et}}(X\setminus Y,\sF) \rightarrow H^{j+1}_{{\text{\'et}},Y}(X,\sF) \rightarrow H^{j+1}_{\text{\'et}}(X,\sF) \rightarrow \cdots.
\end{align*} If $X$ is affine and of finite-type over a separably closed field, then $\ecd(X) = \dim(X)$ \cite[Theorem VI.7.2]{milne1980etale}, and so as soon as $\ecd(X,Y) > \dim(X)$, the sequence once again gives $$\ecd(X \setminus Y) = \ecd(X,Y)-1.$$

As in previous section, let $Y \subseteq X$ be closed subschemes of $\PP^n_\KK$, with affine cones $Y'$ and $X'$, and $U = X \setminus Y$ with affine cone $U' = X' \setminus Y'$. The next proposition is perhaps known to experts, but in lack of a suitable reference, we have here taken the liberty of providing a complete proof.
\begin{proposition} \label{prp:ecd-affine-cone}
We have that $\ecd(U') = \ecd(U) + 1$.
\end{proposition}
\begin{proof}
Lyubeznik has already proved that $\ecd \big(U' \big) \ge \ecd \big(U \big) +1$ (see \cite[Proposition 10.1]{lyubeznik1993etale} and the proof of \cite[Proposition 10.2]{lyubeznik1993etale}).

For the opposite direction, consider the morphism $f:\AA^{n+1} \setminus \{0\} \rightarrow \PP^n$ that takes $(\xi_0,\dots,\xi_n)$ to $[\xi_0 : \dots : \xi_n]$, and let $\pi: U' \rightarrow U$ be its restriction to $U'$. The Leray spectral sequence associated to $\pi$ and to a sheaf $\sF$ on the \'etale site of $U'$ reads
\begin{align*}
H^p \big(U, R^q \pi_*\sF \big) \Longrightarrow H^{p+q} \big(U', \sF \big). 
\end{align*} 

As $f$ is a locally trivial fibration with fiber the punctured affine line $\KK^*$, so is $\pi$; this means that we can cover $U$ by finitely many open sets $V$ such that restricted to $\pi^{-1}(V)$, the morphism $\pi$ is the canonical projection $\pi_V: V \times \KK^* \rightarrow V$. As the open immersion $V \hookrightarrow U$ is a smooth morphism, the smooth base change theorem \citeproject{0EYU} gives 
$$\big(R^q \pi_*\sF\big)|_V = R^q {\pi_V}_* \big(\sF|_{V \times \KK^*} \big),$$ for any sheaf $\sF$ on $U'$. Hence, for any geometric point $\overline{v}$ of $V$, we have 
 $$ \big((R^q \pi_*\sF)|_V\big)_{\overline{v}} = H^q_{\text{\'et}} \left(\Spec\big(\mathcal{O}_{V,v}^{\text{sh}}\big) \times \KK^*, \sF|_{\Spec\big(\mathcal{O}_{V,v}^{\text{sh}}\big) \times \KK^*} \right),$$ where $\mathcal{O}_{V,v}^{\text{sh}}$ is the strict Henselization of the local ring of $V$ at $v$. As the \'etale cohomological dimension of the spectrum of a strictly Henselian  local ring is zero (e.g., see \cite[Proposition 5.7.3]{LeiFu}), and that of $\KK^*$ is $1$ (because it is an affine curve), the K\"unneth formula furnishes 
$$\ecd\Big(\Spec\big(\mathcal{O}_{V,v}^{\text{sh}}\big) \times \KK^* \Big)=1.$$
Hence $(R^q \pi_*\sF)|_V$ is the zero sheaf for $q \ge 2$; as the $V$'s form an open cover of $U$, we infer that $R^q \pi_*\sF=0$ for every $q \ge 2$. With this, the above Leray spectral sequence ensures that $\ecd(U')$ can not exceed $\ecd(U)+1$.
\end{proof}

\subsection{Bounds on Cohomological Dimension}
Faltings has proved the following bound for quasicoherent cohomological dimension:

\begin{theorem}[\cite{faltings1980lokale}, Korollar 2] \label{thm:Faltings}
Let $A$ be a regular local ring that contains a field. Let $I$ be an ideal of big height $c \ge 2$. Then $$\cd(A,I) \le \left \lfloor \left(1-\frac{1}{c} \right) \dim(A) \right \rfloor+1.$$
\end{theorem}

Lyubeznik has proved in \cite{lyubeznik1985some} that the bound of Theorem \ref{thm:Faltings} is tight, and he has also proved a global analogue for \'etale cohomology, which is also tight:

\begin{theorem} [\cite{lyubeznik1993etale}, Theorem 8.2(i)] \label{thm:Lyubeznik-Faltings-analogue}
Let $A$ be a regular ring, of finite type over a separably closed field, and let $I$ be an ideal of big height $c$. Set $X = \Spec(A)$ and $Y = \Spec(A/I)$. Then
 $$\ecd (X,Y) \le 2 \dim(A) - \left \lfloor \frac{\dim(A)-1}{c} \right \rfloor.$$
\end{theorem} 

\begin{remark} \label{rem:Faltings-Lyubeznik}
Expressing $n = \dim(A)$ as $n = qc +r$ with $q$ and $r$ non-negative integers such that $0 \le r < c$, one checks that the upper bounds of Theorems \ref{thm:Faltings} and \ref{thm:Lyubeznik-Faltings-analogue} satisfy  
\begin{align*}
\underbrace{\left \lfloor \left(1-\frac{1}{c} \right) n \right \rfloor+1}_{\text{Theorem \ref{thm:Faltings}}} + n = \underbrace{2 n - \left \lfloor \frac{n-1}{c} \right \rfloor}_{\text{Theorem \ref{thm:Lyubeznik-Faltings-analogue}}}.
\end{align*}
\end{remark}

\subsection{Quasicoherent Cohomological Dimension of Subspace Arrangements}
Let $\KK$ be a separably closed field, $\mathbb{A}_\KK^n$ the affine space of dimension $n$ over $\KK$, and $\V = V_1,\dots,V_s$ a collection of $s$ linear spaces in $\mathbb{A}_\KK^n$, all passing through the origin. With $c_i$ the codimension of $V_i$, we will henceforth be adopting the convention that $$1 \le c_1 \le \cdots \le c_s \le n-1.$$ Denote by $\N$ the nerve complex of $\V$; this is an abstract simplicial complex on vertices $[s]$, such that $F$ is a face if and only if $\cap_{i \in F} V_i$ is non-empty. As all $V_i$'s pass through the origin, $\N$ is the full simplex on $[s]$. For any non-negative integer $d$, we denote by $\N_d$ the subcomplex of $\N$, for which $F$ is a face if and only if $\cap_{i \in F} V_i$ has dimension at least $d$. Note that $\N_{d+1}$ is a subcomplex of $\N_d$. We denote by $X$ the closed subscheme of $\AA_\KK^n$, which is the scheme-theoretic union of the $V_i$'s. 

\begin{theorem}[Pascoe \& Witt \cite{Pascoe-Witt}, Theorem 5.1 and Corollary 5.2] \label{thm:PW}
The quasicoherent cohomological dimension of the complement of $X$ is given by 
\begin{align*}
\qccd(\AA_\KK^n \setminus X) &= n -1 - \min_{d,j} \big[d+j : \, H_j\big(\N_d,\N_{d+1};\KK\big) \neq 0 \big] \\
&= n -1 - \min_{d,j} \big[d+j : \, \tilde{H}_j\big(\N_d;\KK\big) \neq 0 \big].
\end{align*}
\end{theorem}

\subsection{Goresky-MacPherson formula for \'etale cohomology}
Let $W = \cap_{i \in F} V_i$ for some $F \subseteq [s]$. We denote by $\N_W$ the subcomplex of $\N$, for which $A$ is a face of $\N_W$ if and only if $\cap_{i \in A} V_i $ contains the subspace $W$. We denote by $\N_W^\circ$ the subcomplex of $\N_W$, whose faces index intersections that properly contain $W$. 

\begin{theorem}[Yan \cite{Yan}] \label{thm:Yan}
Let $\ell$ be a prime number other than the characteristic of the ground field $\KK$. Then 
$$H_{\text{\'et}}^r\big(\AA_\KK^n \setminus X, \mathbb{Z}/\ell \mathbb{Z}\big)  = 
\bigoplus_{\substack{F \subseteq [s], \\ W = \cap_{i \in F} V_i}} H^{2n-2\dim_\KK(W)-1-r}\big(\N_W,\N_W^\circ; \mathbb{Z}/\ell \mathbb{Z}\big).$$
\end{theorem}

\section{Results}

We say the subspace arrangement $\V$ is linearly general if for any $F \subseteq [s]$ 
$$\codim_\KK \big( \cap_{i \in F} V_i \big) = \min \left\{n, \, \sum_{i \in F} \codim_\KK(V_i) \right\}.$$

\begin{lemma} \label{lem:free}
If $\V$ is linearly general, then $H_j(\N_0,\N_1;\ZZ)$ is a free Abelian group $\forall j$.
\end{lemma}
\begin{proof}
As every $V_i$ has positive dimension, we have $H_0(\N_0,\N_1;\ZZ)=0$; we will thus focus on $H_i(\N_0,\N_1;\ZZ)$ for $i>0$. In this case, since $\N_0$ is a simplex, we have isomorphisms
$$H_i(\N_0,\N_1;\ZZ)\cong H_{i-1}(\N_1;\ZZ), \, \, \, \, \, \, i>1$$ and a short exact sequence
$$0 \rightarrow H_1(\N_0,\N_1;\ZZ) \rightarrow H_0(\N_1;\ZZ) \rightarrow \ZZ \rightarrow 0.$$ As $H_0(\N_1;\ZZ)$ is free, so is its subgroup $H_1(\N_0,\N_1;\ZZ)$. Thus it suffices to prove that $H_i(\N_1;\ZZ)$ is free for $i>0$.

Let us look more carefully at $\N_1$: if $c_i$ is the codimension of $V_i\subseteq \mathbb{A}_\KK^n$, then $F\subseteq [s]$ is a face of $\N_1$ if and only if $\sum_{i\in F}c_i\leq n-1$. This is a shifted complex in the sense of \cite[Definition 11.2]{B-W2}. Thus by \cite[Theorem 11.3]{B-W2} it is shellable, and so by \cite[Theorem 4.1]{B-W1}  (or \cite[Corollary 4.2]{B-W1}) its homology groups are free.
\end{proof}

The following is the main result of this paper.

\begin{theorem} \label{thm:Lyubeznik-linearly-general}
Let $\V$ be a linearly general subspace arrangement. Suppose there is an $F \subseteq [s]$ of minimal cardinality with respect to the property
$$\sum_{i \in F} c_i \le \min \left\{n,\sum_{i \in [s]} c_i\right\}-1 < c_{\alpha}+ \sum_{i \in F} c_i$$
for any $\alpha \in [s] \setminus F$, such that $1 \not\in F$. Then 
$$ \ecd(\AA_\KK^n \setminus X) \ge \dim(\AA_\KK^n \setminus X) + \qccd(\AA_\KK^n \setminus X).$$
\end{theorem}
\begin{proof}
\underline{\emph{Step 1}.}  
Let $R = \KK[x_1,\dots,x_n]$ be the polynomial ring associated to the affine space $\AA_\KK^n$, and $I_i$ the ideal of linear forms corresponding to $V_i$. Then $I = \cap_{i \in [s]} I_i$ is the ideal of $R$ that defines the scheme-theoretic union $X$ of the $V_i$'s. If $\cap_{i \in [s]}V_i \neq 0$, and since $\V$ is linearly general, the $I_i$'s are generated by disjoint variables, which we may take to be $x_1,\dots,x_h$, where $h<n$ is the codimension of $\cap_{i \in [s]}V_i$. There is a partition $\sqcup_{i \in [s]} A_i$ of $[h]$ such that $I_i$ is generated by the variables indexed by $A_i$. With $R'=\KK[x_1,\dots,x_h]$, $I'_i = (x_\alpha: \, \alpha \in A_i)$ and $I' = \cap_{i \in [s]}I_i'$, we have that $I = I' R$. With $X'$ the union of linear spaces in $\AA_\KK^h$ defined by $I'$, we have 
$$ \AA_\KK^n\setminus X = \big(\AA_\KK^h\setminus X' \big) \times_{\KK} \AA_{\KK}^{n-h}.$$ Let $\ell$ be a prime number other than $\cha(\KK)$, such that 
$$ \gamma:=\ecd\big(\AA_\KK^h\setminus X' \big) = \ecd_\ell \big(\AA_\KK^h\setminus X' \big).$$
Thus there exists an $\ell$-torsion abelian sheaf $\sF$ on $\AA_\KK^h\setminus X'$ such that $H_{\text{\'et}}^\gamma\big(\AA_\KK^h\setminus X',\sF\big) \neq 0$. As $\sF$ is the direct limit of its subsheaves $\sF[\ell^\nu]:= \ker(\sF \overset{\ell^\nu}{\rightarrow} \sF)$, and \'etale cohomology commutes with direct limits, we may assume that $\sF$ is an abelian sheaf of $\ZZ/\ell^\nu \ZZ$-modules for some positive integer $\nu$. Using the short exact sequence
$$0 \rightarrow \sF[\ell] \rightarrow \sF[\ell^\nu] \rightarrow \sF[\ell^{\nu-1}] \rightarrow 0$$
and induction, we may further assume that $\sF$ is an abelian sheaf of $\ZZ/\ell \ZZ$-vector spaces. As $\ecd_\ell\big(\AA_\KK^{n-h}) = n-h$, there is, similarly, an abelian sheaf $\sG$ of $\ZZ/\ell \ZZ$-vector spaces on $\AA_\KK^{n-h}$, such that $H^{n-h}_{\text{\'et}}\big(\AA_\KK^{n-h},\sG\big) \neq 0$. Now, the K\"unneth formula \citeproject{0F1P} reads 
$$R\Gamma\big((\AA_\KK^h\setminus X') \times_{\KK} \AA_\KK^{n-h}, p_1^*\sF \otimes_{\ZZ/\ell \ZZ}^{\mathbb{L}} p_2^* \sG \big) = R \Gamma \big(\AA_\KK^h\setminus X',\sF\big) \otimes_{\ZZ/\ell \ZZ}^{\mathbb{L}} R \Gamma \big(\AA_\KK^{n-h},\sG\big),$$ where $p_1$ and $p_2$ are the canonical projections from $(\AA_\KK^h\setminus X') \times_{\KK} \AA_\KK^{n-h}$ to the first and second factor respectively. With $\mathscr{P} = p_1^*\sF \otimes_{\ZZ/\ell \ZZ}^{\mathbb{L}} p_2^* \sG$, the definition of the derived tensor product and the K\"unneth spectral sequence \cite[Theorem 10.90]{rotman} give
$$\underbrace{\bigoplus_{a+b = q} {\Tor}_{p}^{\ZZ/\ell\ZZ}\Big(H_{\text{\'et}}^a\big(\AA_\KK^h\setminus X',\sF\big),H_{\text{\'et}}^b\big(\AA_\KK^{n-h},\sG\big) \Big)}_{E^2_{p,q}} \Rightarrow H_{\text{\'et}}^{p+q}\big((\AA_\KK^h\setminus X')\times_{\KK} \AA_\KK^{n-h}, \mathscr{P} \big).$$ All higher $\Tor$ modules vanish because their computation is over the field $\ZZ/\ell \ZZ$; hence the spectral sequence collapses on the $q$-axis, and so $$ \bigoplus_{a+b = q} H_{\text{\'et}}^a\big(\AA_\KK^h\setminus X',\sF\big) \otimes_{\ZZ/\ell \ZZ} H_{\text{\'et}}^b\big(\AA_\KK^{n-h},\sG\big) = H_{\text{\'et}}^{q}\big((\AA_\KK^h\setminus X')\times_{\KK} \AA_\KK^{n-h}, \mathscr{P} \big).$$ Setting $q = \gamma+n-h$ in the above formula, we obtain
$$H_{\text{\'et}}^{q}\big((\AA_\KK^h\setminus X') \times_{\KK} \AA_\KK^{n-h}, \mathscr{P} \big) = H_{\text{\'et}}^\gamma\big((\AA_\KK^h\setminus X'),\sF\big) \otimes_{\ZZ/\ell \ZZ} H_{\text{\'et}}^{n-h}\big(\AA_\KK^{n-h},\sG\big) \neq 0,$$ as all other direct summands vanish. This shows that 
$$ \ecd \big(\AA_\KK^n\setminus X\big) \ge \ecd \big(\AA_\KK^{h}\setminus X'\big) + n-h.$$ Now, letting $V_i'$ be the linear subspace of $\AA_\KK^h$ defined by the ideal $I_i'$, and $\V'$ the corresponding subspace arrangement, we see that it is enough to prove the statement for $\V'$, because the faithful flatness of the ring map $R'\rightarrow R$ ensures that $\cd(R,I) = \cd(R',I')$, or equivalently
$$ \qccd \big(\AA_\KK^n\setminus X\big)= \qccd \big(\AA_\KK^{h}\setminus X'\big).$$ As $\cap_{i \in [s]}V_i' = 0$, we will henceforth assume that $\cap_{i \in [s]}V_i = 0$. 

\underline{\emph{Step 2}.} We first consider relative simplicial homology $H_{j}\big(\N_{d},\N_{d+1}; \ZZ\big)$ with integral coefficients. Let 
$j^*+d^*$ be minimal such that $H_{j^*}\big(\N_{d^*},\N_{d^*+1}; \ZZ \big) \neq 0$. We will argue that  
$$H_{j^*+d^*}\big(\N_{0},\N_{1}; \ZZ \big) \neq 0.$$
If $d^* = 0$, then there is nothing to argue about, so suppose $d^*>0$. 

Quite generally, for any $j$ and $d$, at homological position $j$ the relative chain complex $C_\bullet(\N_d,\N_{d+1}; \ZZ)$ is supported on those subsets $F \subseteq [s]$ of cardinality $j+1$, for which the dimension of $\cap_{i \in F} V_i$ is exactly $d$. If $d>0$, and as $\V$ is linearly general,  
$$ \dim_\KK \big(\cap_{i \in F'} V_i\big) > \dim_\KK \big(\cap_{i \in F}V_i\big),$$
for any proper non-empty subset $F'$ of $F$. Hence, as already noted by Pascoe \& Witt \cite{Pascoe-Witt}, all differentials of $C_\bullet(\N_d,\N_{d+1})$ are zero, and so
$$ H_j(\N_d,\N_{d+1}; \ZZ) = C_j(\N_d,\N_{d+1}; \ZZ) \, \, \, \text{for} \, \, \, d>0.$$

As $H_{j^*}(\N_{d^*},\N_{{d^*}+1};\ZZ) \neq 0$, we also have that $C_{j^*}(\N_{d^*},\N_{{d^*}+1};\ZZ) \neq 0$. Hence there is an $F \subset [s]$ of cardinality $j^*+1$, such that the dimension of $\cap_{i \in F} V_i$ is $d^*$. Select $\alpha \in [s] \setminus F$ to be maximal. As by hypothesis $\cap_{i \in [s]}V_i =0$ and $\V$ is linearly general, we have 
$$ d':=\dim_\KK \big(\cap_{i \in F \cup \{\alpha \} } V_i\big) < \dim_\KK \big(\cap_{i \in F}V_i\big) = d^*.$$ Now, $F \cup \{\alpha\}$ supports a generator of $C_\bullet(\N_{d'},\N_{d'+1};\ZZ)$ at homological position $j^*+1$. If $d'>0$, the minimality of $j^*+d^*$ implies that $d' = d^*-1$. Then the maximality of $\alpha$ implies that all $V_i$'s with $i \in [s] \setminus F$ are hyperplanes, and proceeding inductively, we may assume that $d^*=1$. Thus at any case, we may assume that $d'=0$ for any $\alpha \in [s] \setminus F$.

We have that $F$ is a facet of the complex $\N_1$. As observed in the proof of Lemma \ref{lem:free}, $\N_1$ is a shifted complex and in particular shellable. By \cite[Theorem 4.3]{B-W1} the reduced homology groups of a shellable complex are supported on the homotopy facets of the complex, i.e. those facets $G$ that satisfy $\mathcal{R}(G) = G$, where $\mathcal{R}(\cdot)$ is the restriction map (see (2.1) in \cite{B-W1}). In particular, the Betti numbers of a shellable complex are obtained by enumerating its homotopy facets. For shifted complexes, \cite[Theorem 11.4]{B-W2} asserts that the restriction map is given by $\mathcal{R}(G) = G \setminus [j]$, where $j$ is minimal such that $j+1 \notin G$. It follows from this characterization that $G$ is a homotopy facet of $\N_1$ if and only if $G$ is a facet of $\N_1$ and $1 \notin G$. So, if $1 \notin F$, then $F$ is a homotopy facet of $\N_1$. If on the other hand $1 \in F$, then $F$ is not a homotopy facet of $\N_1$, but it is facet of $\N_1$ of minimal cardinality (by the minimality of $j^*$). Then our hypothesis implies the existence of a homotopy facet of $\N_1$ of the same cardinality as $F$. At any case, there is a homotopy facet of $\N_1$ of dimension $j^*$, so that $\N_1$ has non-vanishing homology at homological position $j^*$. If $j^*>0$, and recalling that 
$H_j(\N_0,\N_1;\ZZ) \cong H_{j-1}(\N_1;\ZZ)$ for $j>1$, we have that 
$H_{j^*+1}(\N_0,\N_1;\ZZ) \neq 0$ and we are done (note that necessarily $d^*=1$ by the minimality of $j^*+d^*$).

If on the other hand $j^* = 0$, then $F$ is a vertex. As $F$ is a facet of $\N_1$, we have that $\N_1$ is disconnected, and the short exact sequence in the proof of Lemma \ref{lem:free} yields $H_{1}(\N_0,\N_1;\ZZ) \neq 0$ (note that again necessarily $d^*=1$ by the minimality of $j^*+d^*$). 

\underline{\emph{Step 3}.} By Theorem \ref{thm:PW},
$$ \qccd(\AA_\KK^n \setminus X) = n-1-j^\dagger-d^{\dagger},$$
where $j^\dagger+d^\dagger$ is minimal such that $H_{j^\dagger}\big(\N_{d^\dagger},\N_{d^\dagger+1}; \KK\big)$ is non-zero. By the universal coefficient theorem, we have a non-canonically splitting short exact sequence

\begin{tikzcd}
0 \arrow[r] & H_{j}\big(\N_d,\N_{d+1};\ZZ) \otimes_{\ZZ} \KK \arrow[r] & H_{j}\big(\N_d,\N_{d+1};\KK) \arrow[d] &  \\
& & \Tor_1^\ZZ\Big(H_{j-1}\big(\N_d,\N_{d+1};\ZZ), \KK \Big) \arrow[r] & 0,
\end{tikzcd}
from which we see that $j^\dagger + d^\dagger \ge j^*+d^*.$ Hence, $$ \qccd(\AA_\KK^n \setminus X) \le n-1-j^*-d^*.$$

\underline{\emph{Step 4}.} In Theorem \ref{thm:Yan} let us consider the direct summand of 
$$H_{\text{\'et}}^{2n-1-j^*-d^*} \big(\AA_\KK^n \setminus X, \mathbb{Z}/\ell \mathbb{Z}\big)$$ corresponding to $W := \cap_{i \in [s]} V_i = 0$. This summand is 
$$H^{j^*+d^*}\big(\N_W,\N_W^\circ; \mathbb{Z}/\ell \mathbb{Z}\big).$$ As the complex $C_\bullet(\N_W,\N_W^\circ; \mathbb{Z}/\ell \mathbb{Z}\big)$ coincides with $C_\bullet(\N_0,\N_1; \mathbb{Z}/\ell \mathbb{Z}\big)$, and simplicial homology and cohomology are isomorphic over a field, we have that 
$$H^{j^*+d^*}\big(\N_W,\N_W^\circ; \mathbb{Z}/\ell \mathbb{Z}\big) \cong H_{j^*+d^*}\big(\N_0,\N_1; \mathbb{Z}/\ell \mathbb{Z}\big).$$
By Lemma \ref{lem:free}, the abelian group $H_{j^*+d^*}\big(\N_0,\N_1; \ZZ\big)$ is free, and by the definition of $j^*$ and $d^*$ it is non-zero. Hence, $H_{j^*+d^*}\big(\N_0,\N_1; \ZZ\big) \otimes_\ZZ \KK$ is also non-zero. Then, by the short exact sequence of the universal coefficient theorem, $$H_{j^*+d^*}\big(\N_0,\N_1; \mathbb{Z}/\ell \mathbb{Z}\big) \neq 0.$$

\underline{\emph{Step 5}.} By what we have collected so far, it follows that 
\begin{align*}
\ecd(\AA_\KK^n \setminus X) & \stackrel{\text{Step 4}}{\ge} 2n-1-j^*-d^* \\
& \stackrel{\text{Step 3}}{\ge} 2n-1-j^\dagger-d^\dagger \\
& = n + \big(n - 1 - j^\dagger -d^\dagger\big) \\
& \stackrel{\text{Step 3}}{=} \dim\big(\AA_\KK^n \setminus X \big) + \qccd(\AA_\KK^n \setminus X),
\end{align*} and the proof of the theorem is concluded.
\end{proof}

\begin{remark}
From \S \ref{subsection:QCD} and  \S \ref{subsection:ECD}, Theorem \ref{thm:qccd-linearly-general} also implies the inequality
$$ \ecd(\PP_\KK^{n-1} \setminus \bar{X}) \ge \dim(\PP_\KK^{n-1} \setminus \bar{X}) + \qccd(\PP_\KK^{n-1} \setminus \bar{X}),$$ where now $\bar{X}$ denotes $X$ viewed as a projective variety.
\end{remark}

We next characterize $\qccd(\mathbb{A}^n \setminus X)$ in terms of the codimensions $c_1 \le \cdots \le c_s$. First we need some notation.

\begin{definition} \label{dfn:alpha(d)}
Define $\alpha(n) = 0$, and for any integer $1 \le d < n$, define $\alpha(d)$ to be the smallest non-negative integer, such that there are $1<i_1 < \cdots < i_{\alpha(d)} \le s$ with 
\begin{align*}
n-d & \ge c_{i_1} + \cdots + c_{i_{\alpha(d)}}, \\
n-d & < c_1 + c_{i_1} + \cdots + c_{i_{\alpha(d)}}.
\end{align*} \end{definition} 

\noindent As the example below illustrates, $\alpha(d)$ may not exist, while we note that by definition $\alpha(n)$ always exists.  

\begin{example}
Suppose $s=5$ and $n=10$, with $c_1 = 1$ and $c_2 = \cdots = c_5 = 4$. Then $\alpha(0)$ does not exist, because whenever a sum of $c_i$'s with $i>1$ does not exceed $10$, it is less than $10$, so that further adding $c_1$ to it yields a sum that is still less than $10$. 
\end{example}

\noindent As $\alpha(n) = 0$, the minimum in the following statement is taken over a set that is always non-empty.

\begin{theorem} \label{thm:qccd-linearly-general}
With $\V$ linearly general and $\alpha(d)$ as in Definition \ref{dfn:alpha(d)}, we have $$\qccd\big(\AA^n \setminus X \big) = n - \min_{1\le d \le n} \{d+\alpha(d) \},$$ where the minimum is over those $d$'s for which $\alpha(d)$ exists.
\end{theorem}
\begin{proof}
Theorem \ref{thm:PW} gives 
\begin{align*}
\qccd(\AA_\KK^n \setminus X) &= n -1 - \min_{d,j} \big[d+j : \, \tilde{H}_j\big(\N_d;\KK\big) \neq 0 \big].
\end{align*} As $\N_0$ is a simplex, $\tilde{H}_j\big(\N_0; \KK \big)=0$ for every $j$, and so in the above formula we may restrict $d$ to be positive. Furthermore, if $d>n$, $\N_d$ is empty, so it has non-vanishing reduced homology only for $j=-1$ and the expression above on the right-hand-side becomes negative. As $\qccd(\AA_\KK^n \setminus X)$ is always non-negative, it is harmless to further assume $d \le n$, that is
\begin{align*}
\qccd(\AA_\KK^n \setminus X) &= n -1 - \min_{1 \le d \le n, \, j} \big[d+j : \, \tilde{H}_j\big(\N_d;\KK\big) \neq 0 \big].
\end{align*} 

Now, when $d >0$, $\N_d$ is the shifted complex on vertices $[s]$ with $F\subseteq [s]$ a face if and only if $\sum_{i \in F} c_i \le n-d$. Then the existence of a positive $\alpha(d)$ indicates the presence of a homotopy facet of $\N_d$ of dimension $\alpha(d)-1$ (see Step 2 of the proof of Theorem \ref{thm:Lyubeznik-linearly-general}). The minimality of $\alpha(d)$ then implies that $\alpha(d)-1$ is the minimal $j$ such that $\tilde{H}_j\big(\N_d; \ZZ\big) \neq 0$. If on the other hand, $\alpha(d)$ exists but it is zero, then necessarily $\N_d$ is empty, so that again $\alpha(d)-1=-1$ is the smallest homological index where $\tilde{H}_j\big(\N_d; \ZZ\big)$ is non-zero. To summarize, $\alpha(d)-1$ is the smallest $j$ such that $\tilde{H}_j\big(\N_d; \ZZ\big)$ is non-zero, and such a $j$ exists if and only if $\alpha(d)$ exists.

By the universal coefficient theorem, we have a short exact sequence

\begin{tikzcd}
0 \arrow[r] & H_{j}\big(\N_d;\ZZ) \otimes_{\ZZ} \KK \arrow[r] & H_{j}\big(\N_d;\KK) \arrow[d] &  \\
& & \Tor_1^\ZZ\Big(H_{j-1}\big(\N_d;\ZZ), \KK \Big) \arrow[r] & 0.
\end{tikzcd} 

\noindent As $\N_d$ is shellable, its homology is free, so 
$$ H_{j}\big(\N_d;\ZZ) \otimes_{\ZZ} \KK \cong H_{j}\big(\N_d;\KK).$$
Consequently, $\alpha(d)-1$ is the smallest $j$ such that $\tilde{H}_j\big(\N_d; \KK\big)$ is non-zero (if such a $j$ exists) and the above formula becomes
\begin{align*}
\qccd(\AA_\KK^n \setminus X) = n -1 - \min_{1 \le d \le n}  \left\{d+\alpha(d)-1 \right\}  = n - \min_{1 \le d \le n}  \left\{d+\alpha(d) \right\},
\end{align*} where the minimum is over those $d$'s for which $\alpha(d)$ exists.
\end{proof}

The special case of equidimensional linearly general subspace arrangements has already been studied by Pascoe \& Witt \cite{Pascoe-Witt}. While a compact formula for the quasi-coherent dimension can be easily extracted from their Theorem 5.18, here we derive it as a corollary to Theorem \ref{thm:qccd-linearly-general}.

\begin{theorem} \label{thm:qccd-linearly-general-equidimensional}
Let $\V$ be linearly general with $\codim(V_i)=c$ for every $i \in [s]$. Then 
$$\qccd\big(\AA^n \setminus X \big) = \min\left\{n- \left\lceil \frac{n}{c} \right\rceil, s(c-1) \right\}.$$
\end{theorem}
\begin{proof}
First, suppose $sc \le n$. For $d \le n-sc$, $\alpha(d)$ does not exist. For $d>n-sc$, $\alpha(d) = \left\lfloor \frac{n-d}{c} \right\rfloor.$ Now, the function
$d + \left\lfloor \frac{n-d}{c} \right\rfloor$ is increasing in $d$, so its minimum (for $d>n-sc$) is achieved for $d = n-sc+1$. Hence Theorem \ref{thm:qccd-linearly-general} reads
\begin{align*}
\qccd\big(\AA^n \setminus X \big) &= n- \left\{n-sc+1+ \left\lfloor \frac{n-(n-sc+1)}{c} \right\rfloor\right\} \\
&=sc-1-(s-1) = s(c-1).
\end{align*} 

Next, suppose $sc >n$. Then $\alpha(d)=\left\lfloor \frac{n-d}{c} \right\rfloor$ for every $1 \le d \le n$, and Theorem \ref{thm:qccd-linearly-general} reads
\begin{align*}
\qccd\big(\AA^n \setminus X \big) &= n - \min_{1 \le d \le n} \left\{d + \left\lfloor \frac{n-d}{c} \right\rfloor \right\} \\
&= n- \left(1+ \left\lfloor \frac{n-1}{c} \right\rfloor\right) = n-\left\lceil \frac{n}{c} \right\rceil,
\end{align*} where the last equality is easily verified. Finally, one again checks easily that 
$$ n-\left\lceil \frac{n}{c} \right\rceil \le s(c-1)$$
if $sc > n$, and similarly for the other direction.
\end{proof} 

\begin{remark} \label{rem:Faltings-linearly-general}
In the context of Theorem \ref{thm:qccd-linearly-general-equidimensional},  when $sc \ge n$, $\qccd \big(\AA^n \setminus X\big)$ coincides with the bound of Faltings given in Theorem \ref{thm:Faltings}; for this, one just needs to note that $\left\lfloor n- \frac{n}{c} \right\rfloor=n-\left\lceil \frac{n}{c} \right\rceil$. In fact, the union of $\left\lfloor \frac{n}{c} \right\rfloor +1$ linearly general linear spaces of codimension $c$ is a scheme of the same type as the scheme that was used by Lyubeznik to show that Faltings's bound is tight (see \cite[Theorem]{lyubeznik1985some}). 
\end{remark}

The next result, besides providing an explicit formula for the \'etale cohomological dimension for the complement of the union of linearly general equidimensional linear spaces, also serves as an example where Lyubeznik's conjecture holds tightly. 

\begin{theorem} \label{thm:ecd-linearly-general-equidimensional}
Let $\V$ be linearly general, equidimensional subspace arrangement, with $sc \ge n$, where $c = \codim_\KK(V_i)$ for every $i \in [s]$. Then 
$$\ecd\big(\AA^n \setminus X\big) = 2n - \left\lceil \frac{n}{c} \right\rceil = n+
\qccd\big(\AA^n \setminus X\big).$$
\end{theorem}
\begin{proof}
Combining Theorem \ref{thm:Lyubeznik-linearly-general} and Theorem \ref{thm:qccd-linearly-general-equidimensional} we obtain
$$\ecd\big(\AA^n \setminus X\big) \ge 2n - \left\lceil \frac{n}{c} \right\rceil.$$
On the other hand, Theorem \ref{thm:Lyubeznik-Faltings-analogue} asserts the reverse inequality, since $2n-\left\lfloor \frac{n-1}{c} \right\rfloor -1 = 2n - \left\lceil \frac{n}{c} \right\rceil.$
\end{proof}

\begin{remark}
Another way to view Theorem \ref{thm:ecd-linearly-general-equidimensional} is that the complement of an equidimensional linearly general subspace arrangement satisfying $sc \ge n$, is a scheme that simultaneously achieves the upper bounds of Faltings (Theorem \ref{thm:Faltings}) and Lyubeznik (Theorem \ref{thm:Lyubeznik-Faltings-analogue}) for the quasicoherent and \'etale cohomological dimension, respectively. Then by Remark \ref{rem:Faltings-Lyubeznik}, this scheme satisfies Conjecture \ref{conj:lyu} (note that \cite[Theorem 8.5]{lyubeznik1993etale} offers an alternative family of schemes achieving the upper bound of Theorem \ref{thm:Lyubeznik-Faltings-analogue}; see also Remark \ref{rem:Faltings-linearly-general}).
\end{remark}

\bibliographystyle{alpha}
\bibliography{Tsakiris-Varbaro.bbl}

\end{document}